\documentclass[10pt, a4paper]{article}

\usepackage{graphicx}
\usepackage{amsmath}
\usepackage{amsthm}
\usepackage{amsfonts}
\usepackage[utf8]{inputenc}
\usepackage{xcolor}
\usepackage{amssymb}
\usepackage[hyphens]{url}
\usepackage[hidelinks]{hyperref}
\usepackage[hyphenbreaks]{breakurl}
\usepackage{fullpage}
\usepackage{esvect,mathtools}
\usepackage{subfiles}
\usepackage[normalem]{ulem}
\usepackage{enumitem}
\usepackage[numbers,square,sort]{natbib}
\usepackage[capitalize]{cleveref}

\newtheorem{theorem}{Theorem}
\numberwithin{theorem}{section} 
\newtheorem{lemma}[theorem]{Lemma}
\newtheorem{corollary}[theorem]{Corollary}
\newtheorem{question}[theorem]{Question}

\newtheorem{claim}[theorem]{Claim}
\theoremstyle{remark}

\theoremstyle{definition}
\newtheorem{definition}[theorem]{Definition}

\title{Ordered Ramsey numbers of graphs with $m$ edges}
\author{Domagoj Bradač\thanks{Department of Mathematics, ETH Z\"urich, Z\"urich, Switzerland. Research supported in part by SNSF grant 200021-228014. Email: \textbf{\{domagoj.bradac, benjamin.sudakov\}@math.ethz.ch}.} \and Patryk Morawski\footnotemark[2]\thanks{Department of Computer Science, ETH Z\"urich, Z\"urich, Switzerland.
Email: \textbf{pmorawski@student.ethz.ch}.} \and Benny Sudakov\footnotemark[1] \and Yuval Wigderson\footnotemark[3]\thanks{Institute for Theoretical Studies, ETH Z\"urich, Z\"urich, Switzerland. 
   Supported by Dr.\ Max R\"{o}ssler, the Walter Haefner Foundation, and the ETH Z\"{u}rich Foundation. Email: {\textbf{yuval.wigderson@eth-its.ethz.ch}}.}}
\date{}

\begin{document}

\maketitle
\begin{abstract}
    Given a vertex-ordered graph $G$, the ordered Ramsey number $r_<(G)$ is the minimum integer $N$ such that every $2$-coloring of the edges of the complete ordered graph $K_N$ contains a monochromatic ordered copy of $G$.
    Motivated by a similar question posed by Erd\H{o}s and Graham in the unordered setting, we study the problem of bounding the ordered Ramsey number of any ordered graph $G$ with $m$ edges and no isolated vertices.
    We prove that $r_<(G) \leq e^{10^9 \sqrt{m} (\log \log m)^{3/2}}$ for any such $G$, which is tight up to the $(\log \log m)^{3/2}$ factor in the exponent. As a corollary, we obtain the corresponding bound for the oriented Ramsey number of a directed graph with $m$ edges.
\end{abstract}

\section{Introduction}

For a graph $G$, the Ramsey number $r(G)$ is the minimum integer $N$ such that every $2$-coloring of the edges of the complete graph $K_N$ on $N$ vertices contains a monochromatic copy of $G$.
The existence of these numbers was famously proved by Ramsey~\cite{ramsey1930}, while the first good quantitative bounds were proved by Erd\H{o}s and Szekeres~\cite{Erdös1935}. Since then, the field of graph Ramsey theory has flourished, and determining how $r(G)$ depends on the graph $G$ has become one of the most-studied questions in combinatorics.

Arguably, the most important question in the field is determining the Ramsey number $r(K_n)$ of the complete graph $K_n$ on $n$ vertices.
Here, after almost a century of only minor improvements on the standard bounds $2^{n/2} \leq r(K_n) \leq 4^{n}$, a significant breakthrough was recently achieved by Campos, Griffiths, Morris and Sahasrabudhe~\cite{campos2023exponential} who showed an exponentially better upper bound of $(4-\varepsilon)^n$ for some constant $\varepsilon > 0$, see also~\cite{gupta2024optimizing, balister2024upper}.

Another prominent direction of study is to understand the Ramsey numbers of sparse graphs. In 1975, Burr and Erd\H{o}s~\cite{burr_magnitude_1975} conjectured that Ramsey numbers of graphs with bounded degeneracy are linear in their number of vertices. In 1983, Chv\'{a}tal, Rödl, Szemer\'edi and Trotter~\cite{chvatal1983ramsey} proved a special case of the conjecture, namely that bounded degree graphs have linear Ramsey numbers. However, proving the Burr--Erd\H{o}s conjecture in full generality was very challenging, and it was only resolved by Lee in 2017~\cite{lee2017ramsey}.

A related question was posed by Erd\H{o}s and Graham in 1973: among all graphs $G$ on $m$ edges, what graph maximizes the Ramsey number? The intuition given by the above considerations is that one would like to make $G$ as dense as possible. In fact, Erd\H{o}s and Graham~\cite{erdos1975partition} conjectured that among all graphs with $m = \binom{n}{2}$ edges and no isolated vertices it is the complete graph $K_n$ that has the maximum Ramsey number.
This conjecture remains open, and is likely very difficult. Motivated by the lack of progress, in the 1980's Erd\H{o}s~\cite{erdos1984someproblems} asked whether the Ramsey number of any such graph $G$ is at least not much larger that the Ramsey number of the complete graph of the same size.
In other words, he conjectured that there exists a constant $c > 0$ such that for any graph $G$ with $m$ edges and no isolated vertices we have $r(G) \leq 2^{c \sqrt{m}}$.
This conjecture was proved by Sudakov~\cite{sudakov2011conjecture} in 2011.

In this paper, we will study the analogue of the above conjecture for ordered graphs. An ordered graph $G$ on $n$ vertices is a graph whose vertices are labeled with $\{1, \dots, n\}$. We say that an ordered graph $G$ on $[N]$ contains a an ordered graph $H$ on $[n]$ if there exists a mapping $\phi: V(H) \to V(G)$ such that $\phi(i) < \phi(j)$ for each $1 \leq i < j \leq n$ and $(\phi(i), \phi(j)) \in E(G)$ whenever $(i, j) \in E(H)$. For an ordered graph $G$ we then define its ordered Ramsey number $r_<(H)$ as the minimum $N$ such that any $2$-coloring of the complete ordered graph on $[N]$ contains a monochromatic copy of $H$.

The systematic study of ordered Ramsey numbers was initiated by Conlon, Fox, Lee and Sudakov~\cite{conlon2017ordered} and, independently, by Balko, Cibulka, Kr\'al and Kyn\v{c}l~\cite{balko2020ordered}.
Since then it has attracted a considerable amount of interest (e.g.~\cite{balko2020ordered, rohatgi2019off, girao2024ordered, fox202edge}).
In general, ordered Ramsey numbers can behave very differently from their non-ordered counterparts. For example, the Burr--Erd\H{o}s conjecture does not hold for ordered graphs as there exist ordered matchings whose Ramsey number is superpolynomial in their number of vertices~\cite{conlon2017ordered, balko2020ordered}.

In this paper, we initiate the study of the analogue of the conjecture by Erd\H{o}s for ordered graphs.
\begin{question}\label{question:ordered_m_edges}
    Does there exist a constant $c > 0$ such that for any ordered graph $H$ with $m$ edges and no isolated vertices it holds that $r_<(H) \leq 2^{c \sqrt{m}}$?
\end{question}

We believe the answer to this question should indeed be positive. Our main result is a proof of a slightly weaker bound, which differs from the conjectured truth by an additional $(\log \log m)^{3/2}$ factor in the exponent.

\begin{theorem}\label{theorem:m_edges_ordered}
    Let $H$ be an ordered graph with $m$ edges and no isolated vertices. Then
    \[
        r_{<}(H) \leq e^{10^9 \sqrt{m} (\log \log m)^{3/2}}.
    \]
\end{theorem}

In fact, we prove a somewhat stronger statement, namely an off-diagonal version of Theorem~\ref{theorem:m_edges_ordered} in which we may be searching for two different graphs in the two colors;  see Theorem~\ref{theorem:off_diagonal_ordered} for the precise statement.

As a consequence of the off-diagonal result, we immediately get a corresponding theorem for directed graphs. 
For an acyclic directed graph $D$, its oriented Ramsey number, denoted $\vv{r}(D)$, is the minimum integer $N$ such that every tournament on $N$ vertices contains a copy of $D$. Let $D^+$ and $D^-$ be ordered graphs obtained by taking the underlying graph of $D$ and the vertex ordering to be a topological sort of $D$ and its reverse, respectively. 
Fox, He and Wigderson~\cite{fox2024ramsey} observed that $\vv{r}(D) \leq r_<(D^+, D^-)$. 
Theorem~\ref{theorem:off_diagonal_ordered} thus implies the following:

\begin{corollary}
    Let $D$ be an acyclic directed graph with $m$ edges and no isolated vertices. Then
    \[
        \vv{r}(D) \leq e^{10^9 \sqrt{m} (\log \log m)^{3/2}}.
    \]
\end{corollary}

The study of oriented Ramsey numbers was initiated by Stearns in 1951~\cite{stearns1959voting} and since then they have been extensively studied in the literature (e.g.~\cite{fox2024ramsey,draganic2021powers,draganic2022ramsey,kuhn2011aproof,morawski2024oriented}).
Recently, Fox, He and Wigderson~\cite{fox2024ramsey} showed that, as for the ordered graphs, the Burr--Erd\H{o}s conjecture is not true in the oriented setting.
More precisely, they showed that for any $\Delta$ and any $n$ sufficiently large with respect to $\Delta$, there exists an acyclic digraph $D$ on $n$ vertices and with maximum degree $\Delta$ such that $\vv{r}(D) \geq n^{\Omega(\Delta^{2/3}/ {\log^{5/3} \Delta)}}$.
On the other hand, the best known upper bound for the oriented Ramsey number of a digraph $D$ with maximum degree $\Delta$, also due to Fox--He--Wigderson \cite{fox2024ramsey},  is $\vv{r}(D) \leq n^{\mathcal{O}_{\Delta}(\log n)}$.
Thus, there is a big gap between the polynomial lower bound and the super-polynomial upper bound for any fixed $\Delta$.
Here we show that under the weaker assumption that the underlying graph of $G$ is $d$-degenerate for some constant $d \geq 3$, the oriented Ramsey number of $G$ can indeed be super-polynomial.

\begin{theorem}\label{theorem:degenerate_lower_bound}
    For any $n$ there exists a digraph $D$ whose underlying graph is $3$-degenerate such that
    \[
        \vv{r}(D) \geq n^{\Omega(\frac{\log n}{\log \log n})}.
    \]
\end{theorem}
The family of graphs we use to prove \cref{theorem:degenerate_lower_bound} can be viewed as a generalization of subdivisions of tournaments. A $1$-subdivision of a transitive tournament $\overrightarrow{K_n}$ on $n$ vertices is the digraph obtained by taking the set of base vertices $\{1, \dots, n \}$ and for each pair $i < j$ adding one additional vertex $v_{ij}$ together with edges $i v_{ij}$ and $v_{ij} j$.
Recently, oriented Ramsey numbers of $1$-subdivisions of transitive tournaments were studied by various researchers and it was finally proved by Draganić, Munh\'a Correia, Sudakov and Yuster~\cite{draganic2022ramsey} that these numbers are linear in the order of the subdivision.

For our construction instead of pairs of vertices we consider triples. Our diagraph has a set of base vertices $\{1, \dots, n \}$ together with an additional vertex $v_{ijk}$ and edges $iv_{ijk}$, $v_{ijk}j$ and $v_{ijk}k$ for every triple $i < j < k$. This digraph is clearly $3$-degenerate. Somewhat surprisingly, going from pairs to triples increases the oriented Ramsey of such subdivisions from linear to super-polynomial in the number of their vertices. For more details, we refer the reader to \cref{section:subdivisions}.

\vspace{10px}
The remainder of the paper is organized as follows.
In \cref{section:m_edges_proof} we first state the asymmetric version of \cref{theorem:m_edges_ordered}, and, after giving a proof outline, we prove this Theorem.
In \cref{section:subdivisions} we first define a generalization of subdivisions of transitive tournaments and then, using them, we prove \cref{theorem:degenerate_lower_bound}.

\paragraph{Notation and terminology:}
For an ordered graph $G = G_<$, we use $V(G)$ to denote its vertex set and $E(G)$ to denote it edge set.
For $A \subseteq V(G)$, we denote by $G[A]$ the subgraph of $G$ induced by $A$ and write $e_G(A) = |E(G[A])|$.
We define the density $d_G(A)$ of $A$ as $\frac{e_G(A)}{\binom{|A|}{2}}$. 
For the entire vertex set we write $d(G)$ as a shorthand for $d(V(G))$.
For a pair of sets $A, B \subseteq V(G)$, we write $e_G(A, B)$ for the number of edges of $G$ with one endpoint in $A$ and the other in $B$.
We define the density $d_G(A, B)$ between $A$ and $B$ as $\frac{e_G(A, B)}{|A| \cdot |B|}$.
We sometimes drop the subscript and write $d(A)$ instead of $d_G(A)$ etc.\ if the oriented graph is clear from the context.
We write $A < B$ if $a < b$ for all $a \in A$ and $b \in B$.

A \emph{monochromatic book} in a coloring of $E(G)$ is a pair $A, B \subseteq V(G)$ such that all edges in $G[A \cup B]$ with at least one endpoint in $A$ have the same color.
Books have been extensively studied in the Ramsey theory literature (e.g~\cite{rousseau1978ramsey, conlon2022ramsey, fox2023ramsey, campos2023exponential}) and have been an important ingredient in the recent improvements on diagonal Ramsey numbers~\cite{campos2023exponential}.
Throughout this paper, all logarithms are to the base $e$.
We omit floor and ceiling signs whenever they are not essential.

\section{Proof of Theorem \ref{theorem:m_edges_ordered}}\label{section:m_edges_proof}

Instead of proving Theorem \ref{theorem:m_edges_ordered}, we will prove a more general, off-diagonal version of it. For two ordered graphs $H_1, H_2$, we define $r_<(H_1, H_2)$ to be the minimum integer $N$ such that in any red-blue edge-coloring of the complete ordered graph on $N$ vertices, there is a red copy of $H_1$ or a blue copy of $H_2$. We prove the following which clearly implies Theorem \ref{theorem:m_edges_ordered} by taking $H_1 = H_2 = H$.

\begin{theorem}\label{theorem:off_diagonal_ordered}
    Let $H_1,H_2$ be ordered graphs without isolated vertices and with $m_1,m_2$ edges, respectively.
    Then
    \[
        r_<(H_1, H_2) \leq e^{10^8 (m_1m_2)^{1/4} (\log \log (m_1 + m_2))^{3/2}}.
    \]
\end{theorem}

\subsection{Proof outline}
It is natural to attempt proving the above theorem using the approach developed by Sudakov~\cite{sudakov2011conjecture} for the unordered case, which builds on earlier work by Alon, Krivelevich, and Sudakov~\cite{alon2003turan}. However, the ordering of the vertices introduces inherent obstacles that prevent this approach from succeeding directly. To overcome these challenges, we had to introduce new ideas, which we describe below.

To start, let us sketch the approach of finding a copy of an unordered graph $H$ with $m$ edges and no isolated vertices in a coloring of a suitably large clique $K_N$. We shall embed $H$ in $K_N$ in two steps. 
Let $X \subseteq V(H)$ be the set of vertices of $H$ with degree at least $\sqrt{m}$. Note that there are at most $2 \sqrt{m}$ such vertices and that the graph $H'$ obtained from $H$ by removing the vertices $A$ has maximum degree at most $\sqrt{m}$. 
To embed $X$ and $H'$ we want to find a large monochromatic book in $G$.
Since in our case $|V(G)| \geq 2^{c \sqrt{m}}$ for suitably large $c$, we will be able to find a monochromatic, say red, book $(A, B)$, such that $|B| \ge 2 \sqrt{m}$ and $|Y| \ge 2^{c' \sqrt{m}}$ for some large constant $c'$. This monochromatic book can be used to embed $H$, by embedding the vertices $A$ arbitrarily into $X$ and finding a copy of $H'$ in $Y$. Note that by the choice of our sets $X$ and $Y$, any red copy of $H'$ together with the vertices $A$ in $X$ will give us a copy of $H$.

It could happen that there is no red copy of $H'$ in $Y$. In this case, using the greedy embedding technique introduced by Erd\H{o}s and Hajnal~\cite{erdos1989ramsey}, and Graham, R\"{o}dl and Ruci\'{n}ski~\cite{graham200graphs, graham2001bipartite},
one can obtain large disjoint sets $L, R \subseteq Y$ such that most of the edges between $L$ and $R$ are blue. By recursively applying this strategy now inside the sets $L$ and $R$, we find a subset $Y' \subseteq Y$ with very small red density.
The main insight in~\cite{sudakov2011conjecture} is that we can now repeat the same argument on the coloring of the set $Y'$. Crucially, since its blue density is very large, a theorem of of Erd\H{o}s and Szemer\'{e}di~\cite{Erdos1972on} implies that we can find a monochromatic book $(X_2, Y_2)$, in which $|X_2|$ is much larger than $\sqrt{m}$. Thus, we can embed more vertices into $X_2$ and look for a monochromatic copy of $H''$ in $Y_2,$ where now we have a better bound for the maximum degree of $H''$ compared to what we had for $H'$.
It can be shown that the set $Y_2$ is not much smaller than $Y_1$, so by repeating this argument, we eventually find a copy of $H$ at some step, or else a monochromatic clique of size $2m$, which certainly contains $H$.

Let us now return to our setting, where we want to embed an ordered $H$ with $m$ edges into a $2$-colored complete ordered graph on $[N]$. Naively applying the same argument, we can again find a large monochromatic book $(X,Y)$. We can even obtain some control over the ordering, e.g.\ ensuring that $X$ precedes $Y$ in the ordering of $[N]$. However, such a structure will most likely be useless for embedding $H$. Indeed, the high-degree vertices of $H$ do not need to appear consecutively in its given vertex order. Thus, if we try to embed $A$ into $X$, we can not expect to find any copy of $H'$ in $Y$ which  together with vertices in $X$ extends to an ordered copy of $H$.

To overcome this issue, the key new idea is 
instead of finding a monochromatic pair $(X,Y)$ to find a tuple $(A, B_0, \dots, B_{\sqrt{m}})$ of disjoint sets of vertices in $[N]$ such that the following holds. First of all, $(A, B_0 \cup B_1 \cup \dots \cup B_{\sqrt{m}})$ is a monochromatic book. Secondly, we have $|A| = \sqrt{m}$ and $|B_i| \ge b$ for all $i$, for some large parameter $b$. Finally, denoting by $v_1, \dots, v_{\sqrt{m}}$ the vertices of $A$ under the ordering of $[N]$, the ordering of the elements of $A \cup B_0 \cup \dots \cup B_{\sqrt{m}}$ is of the form $B_0, v_1, B_1, v_2, \dots, v_{\sqrt{m}}, B_{\sqrt{m}}$. That is, all the vertices in $B_0$ precede $v_1$, which in turn precedes all the vertices in $B_1$, and so forth. We call this structure a \emph{$(\sqrt{m}, b)$-skeleton}.

Having found such a monochromatic skeleton, say in red, we can now embed all the high-degree vertices of $H$ into $A$, and try to find the remaining part of the graph $H'$ in $B_0 \cup \dots \cup B_{\sqrt{m}}$ in red, this time making sure that each $v \in V(H)$ lands in the correct $B_i$, so as to preserve its relative order with respect to the high-degree vertices. In case we fail, using greedy embedding, we can find a disjoint pair $L, R$ such that most of the edges between $L$ and $R$ are blue. Analogously to the undirected case, we wish to iterate inside $L$ and $R$ to find a set $Y' \subseteq [N]$ which is very dense in blue.

However, here comes the second main difference, which is also the reason for the additional $(\log \log m)^{3/2}$ factor. In the unordered case, after having found a big book $(X,Y)$ and a dense pair $L, R \subseteq Y$, we can simply reapply the greedy embedding argument inside $L$ and $R$, while still using $A$ to embed our high-degree vertices. In the ordered case, however, this is not possible. Indeed, the sets $L,R$ lie entirely within some $B_i,B_j$, respectively, and hence we cannot simply embed all of $H'$ inside one of these parts, as this will not respect the order of the rest of the skeleton. We thus need to find new skeletons inside $L$ and inside $R$. This is very costly and therefore we cannot perform multiple iterations like in the undirected case.

Luckily, we found a way to salvage the situation by performing only two iterations.
In the first one, finding the appropriate skeletons is still cheap, and, similarly to the unordered case, we can continue all the way through until we find a subset $W \subseteq V(G)$ which is dense in one of the colors.
In the second iteration, we then apply the Erd\H{o}s--Szemer\'{e}di theorem~\cite{Erdos1972on} to find a larger skeleton in $W$.
Using this skeleton, we can find a pair $L_2, R_2 \subseteq W$ such that a $1 - \frac{1}{10|V(H)|}$-fraction of the edges between them has one of the colors.
Finally, we inductively find one half of $H$ in $L_2$ and the other half of $H$ in $R_2$.
Given the high edge density between $L_2$ and $R_2$ in one of the colors, we can ensure that these two halves combine to form a complete copy of $H$.

More specifically, we let $N = 2^{C \sqrt{m} \log \log^{3/2} m}$ for some large constant $C$. 
In the first iteration we find a monochromatic, say red, $(s, b)$-skeleton $(X, B_0, \dots, B_s)$ where $s = \sqrt{m \log \log m}$ and $b \ge N / 2^{c s}$ for some constant $c$. 
We then embed $s$ vertices into $X$, and inside $B_0 \cup \dots \cup B_s,$ we try to find a red copy of a graph $H'$ which has maximum degree at most $\Delta = \frac{2m}{s} \le O(\sqrt{m / \log \log m}).$ 
This copy should respect the ordering given by $[N]$. 
In case such a copy doesn't exist, setting $d = 1 / \log^2 m$ and using greedy embedding, we find a pair $(L, R)$ with red density at most $d$ and with $|L|, |R|$ of size roughly $b \cdot d^\Delta$. 
We then rerun the entire argument inside both $L$ and $R$. 
Doing $O(\log d^{-1}) = O(\log \log m)$ such recursive steps, we obtain a set $W$ with density at most $d$ in one of the colors and the size of this set is roughly $N \cdot (2^{-s} d^{\Delta})^{O(\log \log m)} = 2^{C' \sqrt{m} \log \log^{3/2} m}$.

Now, using the Erd\H{o}s--Szemer\'{e}di theorem \cite{Erdos1972on}, inside $W$ we find a monochromatic $(s_2, b_2)$-skeleton with $s_2$ roughly $\log N / d \ge \sqrt{m} \log^2 m$ and $b_2$ still large.
As before, we wish to greedily embed the remaining part of the graph which now has maximum degree $\Delta_2 = O(m / s_2) = O(\sqrt{m} / \log^2 m).$ 
By setting $d_2 = 1 / (10m),$ since $d_2^{\Delta_2} \ge 2^{-\sqrt{m}},$ we obtain a large pair $(L_2, R_2)$ with red, say, density at most $d_2.$ 
Finally, we can inductively find either a red copy of $H$ in $L_2$, in which case we are done, or we can find a blue copy $H_2$ of half of $H$ in $L_2$. 
Since the red density is so small between $L_2$ and $R_2,$ the common blue neighborhood of all the vertices of $H_2$ is large in $R_2$ so it remains to find there a blue copy of the other half in blue or of the whole of $H$ in red, which again follows by induction.

The proof is split into three subsections. In Subsection~\ref{subsec:skeletons} we define skeletons and show how to find them. Then, in Subsection~\ref{subsec:greedy-embedding}, we use these skeletons and the greedy embedding strategy to find sparse pairs and eventually sparse sets in the host graph. Finally, in Subsection~\ref{subsec:putting-together}, we combine these tools to finish the proof of Theorem~\ref{theorem:off_diagonal_ordered}. 

\subsection{Finding large skeletons} \label{subsec:skeletons}
The purpose of this subsection is to define formally skeletons and prove two lemmas which allow us to find them in different situations.
Skeletons are key new ingredient of our proof and  play the same role in the ordered setting that books played in the unordered one. We begin with the definition of an $(a, b)$-skeleton.

\begin{definition}
    Let $a, b$ be positive integers and let $G_<$ be an ordered graph. Let $B = \{v_1, \dots, v_a\} \subseteq V(G_<)$ and $V_0, \dots, V_a \subseteq V(G_<)$. We say that $(B, V_0, V_1, \dots, V_a)$ is an \emph{$(a, b)$-skeleton} if
    \begin{enumerate}[label=\alph*)]
        \item $V_0 < \{v_1\} < V_1 < \{v_2\} < V_2 < \dots < V_{a-1} < \{v_a\} < V_a$; \label{skeleton:order}
        \item $|V_i| \ge b$ for all $0 \leq i \leq a$;
        \item $G_<[B]$ is a clique, and all vertices in $B$ are adjacent to all vertices in $V_0 \cup V_1 \cup \dots \cup V_a$. 
        \label{skeleton:cliques}
    \end{enumerate}
\end{definition}

Now, we will show how to find such a skeleton in a suitable ordered graph $G_<$.
Namely, we will require that for many subsets $V' \subseteq V(G_<)$ of a given size, the induced graph $G_<[V]$ contains a clique of size at least $4a+1$.
This condition then enables us to find an $(a, b)$-skeleton via a simple supersaturation argument.

\begin{lemma}\label{lemma:finding_skeleton}
    Let $N, n, a$ be positive integers satisfying $N \ge n \ge 4a+1$. Let $d \in [0,1]$ and suppose $G = G_<$ is an ordered graph on $N$ vertices such that at least $d \binom{N}{n}$ subsets of size $n$ of $V(G)$ contain a clique of size $4a+1$. Then $G$ contains an $(a, b)$-skeleton with $b = \frac{dN}{n^5}$.
\end{lemma}
\begin{proof}
    Let $A \subseteq (V(G))^{4a + 1}$ be the set of all tuples $(v_0, \dots, v_{4a})$ such that $v_0 < v_1 < \dots < v_{4a}$ and $G[\{v_0, \dots, v_{4a} \}]$ is a clique. We first lower-bound $|A|$ by double counting. Observe that a fixed $(4a+1)$-tuple $X \in A$ is contained in at most $\binom{N - 4a - 1}{n - 4a - 1}$ $n$-element subsets of $V(G)$. Using the assumption, we have
    \[ |A| \ge  d\binom{N}{n} / \binom{N-4a-1}{n-4a-1} \ge d \left( \frac{N}{n} \right)^{4a+1}. \]
    By the pigeonhole principle, there exist vertices $u_1 < u_3 \dots < u_{4a-1} \in V(G)$ such that there are at least $|A| / N^{2a} \ge d \frac{N^{2a+1}}{n^{4a+1}}$ tuples $(v_0, \dots, v_{4a}) \in A$ with $v_1 = u_1, v_3 = u_3, \dots, v_{4a-1} = u_{4a-1}$. Let $A'$ denote the set of all such $(4a+1)$-tuples.

    For each $i = 0, 2, 4, \dots, 4a$ let $V_i$ be the set of all vertices $x$ for which there is a $(4a+1)$-tuple in $A'$ containing $x$ as the $i$th vertex. Note that $|A'| \leq \prod_{i=0}^{2a} |V_{2i}|$ and $|V_i| \leq N$. Therefore, at least $a+1$ of the sets $V_i$ have size at least 
    \[ \left( \frac{|A'|}{N^{a+1}} \right)^{1/a} \ge \left(\frac{d N^a}{n^{4a+1}}\right)^{1/a}  \ge \frac{dN}{n^5} = b. \]
    Therefore, we may choose even indices $0 \leq i_0 < i_1 < \dots < i_a \leq 4a$ such that $|V_{i_k}| \geq b$ for each $k = 0, \dots, a.$ Let $B = \{ u_{i_0+1}, u_{i_1+1}, \dots, u_{i_{a-1}+1}. \}$ We claim that $(B, V_{i_0}, V_{i_1}, \dots, V_{i_a})$ is the desired skeleton. Indeed, we have $|V_{i_j}| \ge b$ for all $j \in [0, a]$ by definition. Furthermore, for each $j \in [0, a]$ and $x \in V_{i_j},$ there is a $(4a+1)$-tuple $(v_0, \dots, v_{4a})$ in $A$ with $v_1 = u_1, v_3 = u_3, \dots, v_{4a-1} = u_{4a-1}$ and $v_{i_j} = x.$ This implies \ref{skeleton:order}~and~\ref{skeleton:cliques}.
\end{proof}

If one of the color classes is very sparse, using the Erd\H{o}s--Szemer\'{e}di theorem \cite{Erdos1972on}, we may find significantly larger skeletons. We first state the Erd\H{o}s--Szemer\'{e}di theorem in the following form with explicit quantitative dependencies.

\begin{lemma}[{Erd\H{o}s--Szemer\'{e}di, e.g.~\cite[Theorem 8.1.4]{yuvalNotes}}]\label{lemma:erdos_szemeredi}
    Let $\varepsilon > 0$ and let $n \ge 1/\varepsilon$ be a positive integer. If $G$ is an $n$-vertex graph with $d(G) \le \varepsilon,$ then $G$ contains a clique or an independent set of size at least $a$, where
    \[ a = \frac{\log n}{100 \varepsilon \log \frac{1}{\varepsilon}}. \]
\end{lemma}

Combining the previous two lemmas, we obtain the following.
\begin{lemma}\label{lemma:skeleton_in_dense_graph}
    Let $c > 0$ and let $a \ge 10 / c$ be a positive integer. Let $G = G_<$ be a complete ordered graph on $N \ge e^{6000 a c \log (c^{-1})}$ vertices with an edge-partition $G_1 \cup G_2$ such that $d(G_1) \leq c$. Then $G_1$ or $G_2$ contains an $(a, b)$-skeleton with 
    \[ b = e^{-6000 a c \log(c^{-1})} \cdot N. \]
\end{lemma}
\begin{proof}
    Set $n = e^{1000 a c \log (c^{-1})}$. Note that $n \ge 1/c$ by assumption on $a$ and observe that by double counting (or Markov's inequality) for at least half of the subsets $V' \subseteq V(G)$ of size $n$, we have $d_{G_1}(V') \leq 2c$. By Lemma \ref{lemma:erdos_szemeredi} for each such set $V'$, the induced subgraph $G_1[V']$ contains either a clique or an independent set (which is a clique in $G_2[V']$) of size $5a$.
    Therefore, for some $i \in \{1,2\}$ for at least a $1/4$-fraction of the subsets $V' \subseteq V(G)$ of size $n$ the induced subgraph $G_i[V']$ contains a clique of size $4a+1$.
    By Lemma \ref{lemma:finding_skeleton} with $d=1/4$, there exists an $(a, b)$-skeleton in $G_i$, where
    \[
        b \geq \frac{dN}{n^5} \ge \frac{1}{4} e^{-5000 a c \log(c^{-1})} \cdot N \ge  e^{-6000 a c \log(c^{-1})} \cdot N,
    \]
    as claimed.
\end{proof}

\subsection{Greedy embedding} \label{subsec:greedy-embedding}
In this section, we prove a greedy embedding lemma which roughly states the following. Let $H$ and $G$ be ordered graphs and for every vertex $v_i$ of $H$ let $V_i$ be some large subset of the vertices of $G$. Then we can either find an embedding $\phi$ of $H$ into $G$ such that $\phi(v_i) \in V_i$ for all vertices $v_i$ of $H$, or we can find a pair $A, B \subseteq V(G)$ such that both $|A|$ and $|B|$ are large and the edge-density between $A$ and $B$ is very low.
The greedy embedding technique was originally developed for the unordered setting (see e.g.~\cite{graham2001bipartite}), and in the ordered setting, a similar lemma was proven in~\cite{conlon2017ordered}.

\begin{lemma}\label{lemma:greedy_embedding}
    Let $0 < c < 1$ and let $H$ be an ordered graph on $n$ vertices, ordered $v_1, \dots, v_n$, with maximum degree at most $\Delta$.
    Additionally, let $G$ be an ordered graph with disjoint non-empty subsets of vertices $V_1, \dots, V_n$ such that $|V_i| \geq N $ for all $i$ and $V_1 < V_2 < \dots < V_n$.
    Suppose there exists no embedding $\phi$ of $H$ into $G$ such that for all $i$ we have $\phi(v_i) \in V_i$.
    Then there exist $A, B \subseteq V(G)$ such that $|A|, |B| \geq (c^{\Delta} / \Delta) N$, $A < B$ and $d(A, B) \leq c$.
\end{lemma}
\begin{proof}
    We will attempt to find such an embedding $\phi$ of $H$ into $G$ using the greedy embedding technique.
    Since we are doomed to fail, this process will have to get stuck at some point, which will give us our dense pair.

    For $0 \leq t < i \leq n$ let $N_t(v_i) = N_H(v_i) \cap \{ v_1, \dots, v_t \}$.
    We start by setting $U^{(0)}_i = V_i$ for each $v_i \in V(H)$ and inductively pick $\phi(v_i)$ in the order $v_1, \dots, v_n$.
    At each step $t$, we will keep track of the valid candidates $U_i^{(t)}$ for the vertices where we can still put $v_i$.
    We make sure that they satisfy the following properties:
    \begin{enumerate}
        \item For each $i, t \in [n]$ we have $U_i^{(t)} \subseteq U_i^{(t-1)} \subseteq V_i$,
        \item For each $1 \leq i \leq t \leq n$ we have $U_i^{(i)} = \{ \phi(v_i) \}$,
        \item For every $0 \leq t < i \leq n$ we have $|U_i^{(t)}| \geq c^{|N_t(v_i)|} |V_i|$, and
        \item For every $1 \leq i \leq n$ and $t, j \geq i$ if $v_iv_j \in E(H)$ then $\phi(v_i)x \in E(G)$ for every $x \in U_j^{(t)}$.
    \end{enumerate}
    For $t = 0$, the conditions are clearly satisfied.
    Moreover, if we can find such sets all the way up to step $t = n$, then we have found an embedding $\phi$ of $H$ into $G$ such that $\phi(v_i) \in V_i$ for all $i \in [n]$.

    We attempt to make each step in the following way.
    Suppose that we succesfully continued our process until some step $t - 1$.
    We try to find a $w_t \in U_t^{(t-1)}$ such that for every $i > t$ with $v_tv_i \in E(H)$ we have $|N_G(w_t) \cap U_i^{(t-1)}| \geq c |U_i^{(t-1)}|$.
    Then, we can set $\phi(u_t) = w_t$, $U_t^{(t)} = \{ w_t \}$ and for $i \neq t$
    \[
        U_i^{(t)} = \begin{cases}
            U_i^{(t-1)} \cap N_G(w_t) &\text{if } v_tv_i \in E(G),\\
            U_i^{(t-1)} & \text{otherwise}.
        \end{cases}
    \]
    Then, for those $i > t$ with $v_tv_i \in E(G)$ we have $|N_t(v_i)| = |N_{t-1}(v_i)| + 1$ and thus
    \[
        |U_i^{(t)}| \geq c |U_i^{(t-1)}| \geq c^{|N_t(v_i)|} |V_i|.
    \]
    For the remaining choices of $i$ we have $|N_t(v_i)| = |N_{t-1}(v_i)|$ and thus also
    \[
        |U_i^{(t)}| =  |U_i^{(t-1)}| \geq c^{|N_t(v_i)|}|V_i|.
    \]
    Since the new sets $U_i^{(t)}$ clearly also satisfy the other properties, we could continue the process up through step $t$.

    Since the process cannot continue up through $t = n$, at some step $1 \leq t < n$ we must have that for each $w \in U_t^{(t)}$ there exists some $i$ such that $v_tv_i \in E(H)$ but $|N_G(w) \cap U_i^{(t-1)}| < c|U_i^{(t-1)}|$.
    Since $v_i$ has at most $\Delta$ neighbors, by the pigeonhole principle there exists some $i$ such that $v_tv_i \in E(H)$ and the set $A$ of all $w \in U_t^{(t-1)}$ with less than $c |U_i^{(t-1)}|$ neighbors in $U_i^{(t-1)}$ has size at least $|A| \geq |U_t^{(t-1)}| / \Delta$.

    Now, set $B = U_i^{(t-1)}$ and notice that $d(A, B) \leq c$.
    Moreover, since for all $j > t-1$ we have $|N_{t-1}(v_j)| \leq |N_H(v_j)| \leq \Delta$ we get that
    \[
        |A| \geq c^{|N_{t-1}(v_t)|} |V_t| / \Delta \geq (c^\Delta / \Delta) N
    \]
    and
    \[
        |B| \geq c^{|N_{t-1}(v_i)|} |V_i| \geq c^\Delta N> (c^\Delta / \Delta) N,
    \]
    as desired.
\end{proof}

If we are given a large skeleton, we apply the previous lemma. Doing so yields the following result.
\begin{lemma}\label{lemma:using_skeleton}
    Let $c \in (0,1)$ and let $a, b, m$ be positive integers satisfying $b \geq 2m^2 c^{-\frac{2m}{a}}$. Let $G = G_<$ be an ordered graph with an $(a, b)$-skeleton $(F, V_0, \dots, V_a)$ and let $H$ be an ordered graph with at most $m$ edges and no isolated vertices. If $G$ contains no copy of $H$, then there exist $A, B \subseteq V(G)$ such that $|A|, |B| \geq c^{\frac{2m}{a}} \cdot \frac{b}{2m^2}$, $A < B$ and $d(A, B) \leq c$.
\end{lemma}
\begin{proof}
    Let $V(H) = \{ u_1, \dots, u_n \}$ such that the ordering of $H$ is $u_1, \dots, u_n$ and note that $n \leq 2m$ since $H$ has no isolated vertices.
    Let $1 \leq i_1 < \dots < i_a \leq n$ be the indices of the $a$ vertices of $H$ with the largest degree in $H$ and let $H'_\prec = H \setminus \{u_{i_1}, \dots, u_{i_a}\}$.
    Note that the $\Delta \coloneqq \Delta(H') \le \frac{2m}{a}.$

    For each $j \in [0,a]$ we partition the set $V_j$ equally into at most $n$ sets $V'_{i_j + 1}< \dots< V'_{i_{j+1} - 1}$, where, for convenience, we set $i_0 = 0$ and $i_{a+1} = n + 1$.
    Note that for each $i \in [n] \setminus \{i_1, \dots, i_a\}$ we have $|V'_{i}| \geq \frac{b}{n}$.

    Let $F = \{v_1, \dots, v_a \}$ such that $v_1 < \dots < v_a$ and suppose that we can find an embedding $\phi$ of $H'$ into $G$ such that for each $u_i \in V(H')$ we have $\phi(u_i) \in V'_{i}$.
    Then, by the definition of an $(a,b)$-skeleton, we can set $\phi(u_{i_k}) = v_k$ for each $k \in [a]$ to obtain an embedding of $H$ into $G$.
    Thus, such an embedding $\phi$ cannot exist. Therefore, by Lemma \ref{lemma:greedy_embedding} there exist $A, B \subseteq V(G)$ such that $|A|, |B| \geq (c^\Delta / \Delta) \frac{b}{n} \ge c^{\frac{2m}{a}}\frac{b}{2m^2}$, $A < B$ and $d(A, B) \leq c$.    
\end{proof}

Finally, we can find a skeleton, apply greedy embedding to find a sparse pair $(A, B)$ and recursively repeat the argument inside each of $A$ and $B$ to eventually obtain a sparse set.

\begin{lemma}\label{lemma:recursive_pairs}
    Let $m_1, m_2$ be positive integers with $m_1 \ge \max\{m_2, 100\}$ and let $c \in (0,1/8).$ Let $H_1$ and $H_2$ be ordered graphs with $m_1$ and $m_2$ edges, respectively, and no isolated vertices. Suppose that the edges of the complete ordered graph $G$ on $N$ vertices are partitioned into two ordered graphs $G_1, G_2$. Suppose $H_i$ is not a subgraph of $G_i,$ for $i \in [2].$ Then there exists an $i \in [2]$ and a set set $W \subseteq V(G)$ satisfying $d_{G_i}[W] \le c$ and
    \[ |W| \ge \exp \left(-500 \log(c^{-1}) \left(\log(c^{-1}) \sqrt{\frac{m_2}{\log \log m_1}} + \log \left( \frac{2m_1}{m_2} \right) \sqrt{m_2 \log \log m_1} \right)\right) \cdot N. \] 
\end{lemma}
\begin{proof}
    Let $k_2 = \sqrt{m_2 \log \log m_1}$ and $k_1 = m_1 \cdot \frac{k_2}{m_2}\geq k_2$. Observe that $\frac{m_1}{k_1} = \frac{m_2}{k_2} = \sqrt{\frac{m_2}{\log \log m_1}}$.
    Moreover, let
    \[
     n = r(5k_1, 5k_2) \leq \binom{5k_1 + 5k_2}{5k_2} \leq \binom{10k_1}{5k_2} \le \left( \frac{6k_1}{k_2} \right)^{5k_2} \le e^{30 k_2 \log(\frac{e \cdot k_1}{k_2})} \le e^{30 \sqrt{m_2 \log \log m_1} \log (\frac{e \cdot m_1}{m_2})}.
    \]

    The lemma easily follows from the following claim.
    \begin{claim} \label{cl:binary-tree}
        Set $\alpha = \left(\frac{c}{8}\right)^{2\sqrt{\frac{m_2}{\log \log m_1}}} / (8n^5m_1^2)$. Let $h_1, h_2$ be nonnegative integers and let $X \subseteq V(G)$ be a nonempty set of vertices. Then, for some $i \in [2],$ there exists a set $W \subseteq X$ of size at least $\alpha^{h_1 + h_2} |X|$ such that $d_{G_i}[W] \le 2^{-h_i} + c/2.$
    \end{claim}

    Before proving the claim, let us finish the proof of the lemma given Claim~\ref{cl:binary-tree}. By applying this claim with $h = h_1 = h_2 = \lceil \log_2(2/c) \rceil$ and $X = V(G)$ we get an $i \in [2]$ and a set $W$ with $d_{G_i}[W] \le 2^{-h} + c/2 \le c$ and $|W| \ge \alpha^{-2h} N.$ It remains to verify that $W$ is large enough. Note that 
    \[ 8n^5 m_1^2 \le e^{160 \sqrt{m_2 \log \log m_1} \log\left(\frac{e \cdot m_1}{m_2}\right)}, \]
    where we used that $m_1 \ge 100.$ Therefore,
    \begin{align*}
        |W|/N &\ge \alpha^{2h} \ge \exp\left(-2 \log(3c^{-1}) \left( 4 \log (c^{-1}) \sqrt{\frac{m_2}{\log \log m_1}} + 160 \sqrt{m_2 \log \log m_1} \log\left(\frac{e \cdot m_1}{m_2}\right) \right)\right)\\
        &\ge \exp\left( -500 \log(c^{-1}) \left(\log(c^{-1})\sqrt{\frac{m_2}{\log \log m_1}} +\log\left(\frac{e \cdot m_1}{m_2}\right) \sqrt{m_2 \log \log m_1} \right) \right),
    \end{align*} 
    as needed. \qedhere
    
    \begin{proof}[Proof of Claim~\ref{cl:binary-tree}]
        We will prove the statement by induction on $h_1 + h_2$. If $h_j = 0$ for some $j \in [2],$ then the claim trivially holds by taking $i=j$ and $W = X$. Now, assume $h_1, h_2 > 0$ and that the claim holds for all $h_1' + h_2' < h_1 + h_2.$ Furthermore, note that we may assume $|X| \ge \alpha^{-(h_1 + h_2)},$ as otherwise the claim is fulfilled by taking $W$ to consist of a single vertex.
                        
        By the definition of $n =r(5k_1,5k_2)$, we know that for every $Y \subseteq X$ of size $n$, $G_1[Y]$ contains a $(4k_1+1)$-clique or $G_2$ contains a $(4k_2+1)$-clique. By the pigeonhole principle, there is an $i \in [2],$ such that for at least $\frac{1}{2} \binom{|X|}{n}$ sets $Y$, $G_i$ contains a $(4k_i + 1)$-clique. 
        
        By Lemma~\ref{lemma:finding_skeleton} there is a $\left(k_i, \frac{|X|}{2n^5}\right)$-skeleton in $G_i$. Furthermore, since $G_i$ contains no copy of $H_i$, by Lemma \ref{lemma:using_skeleton} there are sets $A, B \subseteq X$ with $A < B$ such that $d_{G_i}(A, B) \leq \frac{c}{8}$ and
        \[
            |A|, |B| \geq \left(\frac{c}{8}\right)^{\frac{2m_i}{k_i}} \frac{|X|}{4n^5 m_i^2} \geq 2 \alpha |X|.
        \]
        Let $A' \subseteq A$ be the $\alpha |X| \leq |A|/2$ vertices in $A$ with the lowest degree into $B$ and note that each vertex in $A'$ has at most $\frac{c}{4} |B|$ neighbors in $B$.
        We apply the induction hypothesis with $h_i' = h_i - 1$ and $h_{3-i}' = h_{3-i}$ on the induced subgraph $G[A']$. Thus for some $\ell \in [2],$ there is a set $W_1' \subseteq A'$ of size at least $\alpha^{h_1' + h_2'} |A'| = \alpha^{h_1 + h_2} |X|$ with $d_{G_\ell}[W_1'] \le 2^{-h'_\ell} + c/2.$ If $\ell \neq i,$ then we are done since $h_\ell' = h_\ell.$ So assume that $\ell = i$. By averaging, there is a subset $W_1 \subseteq W_1'$ of size exactly $\alpha^{h_1 + h_2} |W|$ with $d_{G_i}[W_1] \le 2^{-h_i+1} + c/2$.

        Observe that $d_{G_i}(W_1, B) \le c/4$ since in $G_i$ every vertex in $A' \supseteq W_1$ has at most $\frac{c}{4} |B|$ neighbors in $B$. Let $B'$ be the set of $\alpha |X| \le |B| / 2$ vertices with the lowest degree in $G_i$ into the set $W_1$. Then in $G_i$ every vertex in $B'$ has at most $\frac c2 |W_1|$ neighbors in $W_1$.
            
        We apply the induction hypothesis on the graph $G[B']$ with $h'_i = h_i - 1$ and $h_{3-i}' = h_{3-i}$. Again, if we find a sparse set in $G_{3-i},$ we are done, so we assume that there is a set $W_2' \subseteq B'$ of size at least $\alpha^{h_1' + h_2'} |B'| \ge \alpha^{h_1 + h_2} |X|$ with $d_{G_i}[W_2'] \le 2^{-h_i + 1} + c/2$. Again, by averaging there is a subset $W_2 \subseteq W_2'$ of size exactly $\alpha^{h_1 + h_2} |X|$ with $d_{G_i}[W_2] \le 2^{-h_i+1} + c/2$.

        We claim that $W_1 \cup W_2$ is the desired set. Indeed, recall that in $G_i$ every vertex in $W_2$ has at most $\frac{c}{2} |W_1|$ neighbors in $W_1.$ Therefore, since $|W_1| = |W_2|,$ we have
        \[ d_{G_i}[W_1 \cup W_2] = \frac{1}{4} \left( d_{G_i}[W_1] + d_{G_i}[W_2]\right) + \frac{1}{2} d_{G_i}[W_1, W_2] \le \frac{1}{4} \cdot 2 \cdot (2^{-h_i+1} + c/2) + \frac{1}{2} \cdot c/2 = 2^{-h_i} + c/2, \]
        as required.
    \end{proof}
\end{proof}

\subsection{Putting things together} \label{subsec:putting-together}
We are now ready to prove Theorem \ref{theorem:off_diagonal_ordered}.

\begin{proof}[Proof of Theorem \ref{theorem:off_diagonal_ordered}]
    We prove the statement by induction on $m_1 \cdot m_2$.
    For $m_1 \cdot m_2 \leq 10^6$ the statement clearly holds, since a colored complete ordered graph on $2^{2\cdot10^8}$ vertices contains a clique of size $10^8$ in one of the two colors.

    Now let $m_1, m_2 \in \mathbb{N}$ and suppose that the statement holds for all $m_1'm_2' < m_1 m_2$.
    Without loss of generality, suppose that $m_1 \geq m_2$ and let $H_1$ and $H_2$ be ordered graphs with no isolated vertices and with $m_1$ and $m_2$ edges respectively.
    Moreover let $G$ be a complete ordered graph on $N = e^{10^8 (m_1m_2)^{1/4} (\log \log (m_1 + m_2))^{3/2}}$ 
    vertices whose edges are colored red and blue; we let $G_1$ and $G_2$ be the red and blue graphs respectively. 
    Suppose for contradiction that there is neither a copy of $H_1$ in $G_1$ nor a copy of $H_2$ in $G_2$.

    We first let $c_1 = \frac{m_2}{m_1 \log^2 m_1}$.
    By \cref{lemma:recursive_pairs} we can find an $i_1 \in [2]$ and $W \subseteq V(G)$ such that $d_{G_{i_1}}(W) \leq c_1$ and
    \begin{align*}
        \frac{|W|}{N} &\geq \exp\left(-500 \log \left(\frac{m_1 \log^2 m_1}{m_2}\right)\left(\log \left(\frac{m_1 \log^2 m_1}{m_2}\right) \sqrt{\frac{m_2}{\log \log m_1}} + \log\left(\frac{e \cdot m_1}{m_2}\right) \sqrt{m_2 \log \log m_1} \right) \right)\\
        &\geq \exp \left( -10^4\sqrt{m_2} \cdot (\log \log m_1)^{3/2} \cdot \log^2 \left(\frac{e\cdot m_1}{m_2}\right)  \right) \\
        & \geq \exp \left(  -10^5 \cdot (m_1 m_2)^{1/4}\cdot (\log \log m_1)^{3/2} \right),
    \end{align*}
    where in the first inequality we use twice that $\log \left(\frac{m_1 \log^2 m_1}{m_2} \right) \leq 2 \cdot \left(\log(\frac{e \cdot m_1}{m_2}) \right) \cdot \left(\log \log m_1\right).$
    Further, let $a = \frac{10 m_1}{\sqrt{m_2}}\log^2 m_1$ and notice that since $a \geq \frac{10}{c_1}$ by Lemma \ref{lemma:skeleton_in_dense_graph} there is a $i_2 \in [2]$ such that $G_{i_2}[W]$ contains an $(a, b)$-skeleton for
    \begin{align*}
        b & = |W| \cdot \exp \left( -6000 \cdot  \frac{10 m_1 \log^2 m_1}{\sqrt{m_2} }\cdot \frac{m_2 }{m_1 \log^2 m_1} \cdot \log \left( \frac{m_1 \log^2 m_1}{m_2}\right)\right) \\
        &\geq |W| \cdot \exp \left( -10^6 \cdot \sqrt{m_2} \cdot \log \log m_1 \cdot \log \left(\frac{e \cdot m_1}{m_2}\right)\right) \\
        & \geq \exp \left( (10^8 - 10^6 - 10^5)(m_1 m_2)^{1/4} (\log \log (m_1 + m_2))^{3/2} \right).
    \end{align*}
    
    We now let $c_2 = \frac{1}{6m_1}$.
    Notice that $b \geq 2 m_{1}^2 c_2^{-\frac{2m_{1}}{a}} \geq 2m_{i_2} c_2^{- \frac{2 m_{i_2}}{a}}$ and therefore by Lemma \ref{lemma:using_skeleton} we can find $A, B \subseteq V(G)$ such that $A < B$, $d_{G_{i_2}}(A, B) \leq c_2$ and 
    \begin{align*}
        |A|, |B| & \geq \frac{b}{2m_{i_2}^2} \cdot \exp \left( - \log(6m_{i_2}) \cdot 2m_{i_2} \cdot \frac{\sqrt{m_2}}{10 m_1 \log^2 m_1} \right) \\
        & \geq \exp \left( (10^8 - 10^6 - 10^5 - 10)(m_1 m_2)^{1/4} (\log \log (m_1 + m_2))^{3/2} \right).
    \end{align*}

    We now let $i_3 = 3 - i_2$ and notice that $d_{G_{i_3}}(A, B) \geq 1 - c_2$.
    We want to use this dense pair for our inductive step.
    To that end, let $n = |V(H_{i_3})| \leq 2m_1$ and let $v_1 < \dots < v_n$ be the vertices of $H_{i_3}$.
    Moreover, let $\ell \in [n]$ be the largest index such that for $U_L = \{ v_1, \dots, v_\ell \}$ we have $|E(H_{i_3}(U_L))| \leq m_{i_3} / 2$.
    Let $U_R = V(H_{i_3}) \setminus U_L$ and notice that $|E(H_{i_3}[U_R])| \leq m_{i_3} / 2$ as well.
    Let $L$ and $R$ be the graphs obtained by removing the isolated vertices from $H_{i_3}[U_L]$ and $H_{i_3}[U_R]$ respectively.
    
    Now, let $A'$ be the vertices in $A$ with at least $(1-2c_2)|B|$ neighbors in $B$ and notice that we have $|A'| \geq |A| / 2$.
    Moreover, let $A''$ be the subset of $A'$ obtained by taking every $3m_{1}$-th vertex of $A$ under the ordering of $G$, where we also omit the last vertex we would add to $A''$ in this process.
    Since
    \begin{align*}
        |A''| & \geq  \frac{1}{6m_1^2} \exp\left( (10^8 - 10^6 - 10^5 - 10)(m_1 m_2)^{1/4} (\log \log (m_1 + m_2))^{3/2} \right) - 1 \\
        & \geq \exp \left( 10^8 \left(\frac{m_1\cdot m_2}{2}\right)^{1/4} (\log \log (m_1 + m_2))^{3/2} \right),
    \end{align*}
    by the induction hypothesis we can find an ordered copy of $H_{i_2}$ in $G_{i_2}[A'']$ or an ordered copy of $H_{i_3}$ in $G_{i_3}[A'']$.
    In the former case, we are done.
    Let us therefore assume that the latter happens, i.e., we find an embedding $\phi_1$ of $L$ into $G_{i_3}[A'']$.
    Note that the number of isolated vertices in $H_{i_3}[U_L]$ is at most $2m_{i_3} < 3m_1$ and therefore, since between any two vertices in $A'$ that are used by $\phi$ there is at least $3m_1$ free vertices, we can extend $\phi_1$ into an embedding $\phi_2$ of $H_{i_3}[U_L]$ into $A'$.

    We now let $B' \subseteq B$ be the set of common neighbors of $\phi_2(U_L)$ in $B$ and notice that by our choice of $A'$, we have $|B'| \geq |B| - 2m_{i_3} \frac{1}{3m_1} |B| \geq |B| / 2$.
    We define $B'' \subseteq B'$ in the same way as $A''$ above, i.e., as the set obtained by taking every $3m_1$-th vertex of $B'$ and omitting the last vertex we would add in such a manner.
    Then we again have
    \begin{align*}
        |B''| & \geq  \frac{1}{6m_1^2} \exp\left( (10^8 - 10^6 - 10^5 - 10)(m_1 m_2)^{1/4} (\log \log (m_1 + m_2))^{3/2} \right) - 1 \\
        & \geq \exp \left( 10^8 \left(\frac{m_1\cdot m_2}{2}\right)^{1/4} (\log \log (m_1 + m_2))^{3/2} \right),
    \end{align*}
    and thus, by the induction hypothesis we can either find a copy of $H_{i_2}$ in $G_{i_2}[B'']$ or a copy of $H_{i_3}[U_R]$ in $G_{i_3}[B'']$.
    In the former case we are done, in the latter case we can again obtain an embedding $\phi_3$ of $H_{i_3}[U_R]$ into $G_{i_3}[B']$.

    We now let $\phi: V(H_{i_3}) \to V(G_{i_3})$ be defined as
    \[
        \phi(v) = \begin{cases}
            \phi_2(v), & v \in U_L \\
            \phi_3(v), & v \in U_R
        \end{cases}.
    \]
    The, since $A' < B'$ and $ab \in E(G_{i_3})$ for all $a \in A'$, $b \in B'$ we have that $\phi'$ is an embedding of $H_{i_3}$ into $G_{i_3}$, which concludes the proof.
\end{proof}

\section{Larger subdivisions}\label{section:subdivisions}
In this section, we prove \cref{theorem:degenerate_lower_bound} using the following digraph, which we formally call a $(1, 2)$-subdivision of a transitive tournament.
\begin{definition}
    A $(1, 2)$-subdivision $S_n = (V, E)$ of the transitive tournament on $n$ vertices is the acyclic digraph with the vertex set
    \[
    V = [n] \cup \{ (i, j, k) \in [n]^3 : i < j < k\}
    \]
    and the edge set
    \[
        E = \{ (i, (i,j,k)), ((i, j, k), j), ((i,j,k), k) : 1 \leq i < j < k \leq n \}.
    \]
    We call the set $[n]$ the base vertices of $S_n$.
\end{definition}
It is easy to check that $(1, 2)$-subdivisions are $3$-degenerate. Next we prove that they have super-polynomial oriented Ramsey numbers by constructing a suitably large host tournament which does not contain a copy of $S_n$.
Specifically, we take the iterated blow-up of a random tournament.
We argue that since the random tournament will not contain a transitive tournament on $4 \log n$ vertices, we will be able to use at most $2 \log n$ of the blobs in any embedding of $S_n$.
In particular, in some of the blobs we would have to find a copy of $S_{n'}$, which we exclude by construction.

\begin{theorem}\label{theorem:1_2_subdivision_ramsey}
    For each $n \geq 3$ we have
    \[
    \vv{r}(S_n) \geq n^{\log n /100 \log \log n}
    \]
\end{theorem}

\begin{proof}
    We prove the statement by induction on $n$. The base case is trivial, since for $n \leq 20$ we clearly have $\vv{r}(S_{n}) \geq 4 \geq 20^{\log 20 / 100\log \log 20}$.

    Now suppose that for some $n \in \mathbb{N}$ the statement holds for all $n' < n$.
    We aim to construct a tournament $T$ on $|V(T)| \geq n^{\log n /100 \log \log n}$ vertices that doesn't contain a copy of $S_n$.

    Let therefore $R$ be a tournament on the vertex set $[m]$ where $m = n/10$ that doesn't contain a copy of a transitive tournament on $4\log n$ vertices.
    Note that the probability that a uniformly random tournament on $n_1$ vertices contains a copy of such a transitive tournament is at most $m^{4 \log n} 2^{-8\log^2 n} < 1$, and thus such a tournament indeed exists.

    Now, let $n' = \frac{n}{40\log n}$ and let $T'$ be a tournament on $|V(T')| = n'^{\log n' /100 \log \log n'}$ vertices that contains no copy of $S_{n'}$, which exists by induction.
    We let $T$ be a blow-up of $R$ obtained by replacing each of its vertices by a copy of $T'$.
    More formally, we let $N = m \cdot n'^{\log n' /100 \log \log n'}$ and $T$ be the tournament on the vertex set $V(T) = [N] = V_1 \cup \dots \cup V_{n_1}$, where $|V_\ell| = |V(T')|$ for all $\ell$, defined as follows. For each $\ell \in [m]$, $T[V_\ell]$ is a copy of $T'$, and for each $ij \in E(R)$ we have $V_i \times V_j \subseteq E(T)$.
    
    Suppose now that there is a copy $D$ of $S_n$ in $T$.
    For each $\ell \in [m]$ let $B_\ell$ be the set of the base vertices embedded into $V_\ell$.

    \begin{claim}
        For each $\ell \in [m]$ we have $|B_\ell| < n'$.
    \end{claim}
    \begin{proof}
        Suppose that $|B_\ell| \geq n'$ holds for some $\ell \in [m]$.
        We will show that in this case $T[V_\ell]$ contains a copy of $S_{n'}$.

        Indeed, let $v = (i, j, k) \in V(T) \cap B_\ell^3$ and notice that since $i \to v$, $v \to j$ and $\phi(i), \phi(j) \in V_\ell$, we must have that $\phi(v) \in V_\ell$. Indeed, if $\phi(v)$ is in any other part, then the edges $\phi(i)\phi(v)$ and $\phi(j)\phi(v)$ must be oriented identically.
        Therefore, with a slight abuse of notation, $\phi(V[B_\ell \cup B_\ell^3]) \subseteq V_\ell$ and, since $V[B_\ell \cup B_\ell^3]$ is isomorphic to $S_{|B_\ell|}$ and $|B_\ell| \geq n'$, we get that $T[V_\ell]$ contains a copy of $S_{n'}$, a contradiction to $T[V_k]$ being a copy of $T'$.
    \end{proof}

    \begin{claim}
        We have $|\{ \ell \in [m]: |B_\ell| \geq 2 \}| < 4\log n$.
    \end{claim}
    \begin{proof}
        Without loss of generality, suppose that $\{ i \in [m]: |B_\ell| \geq 2 \} = [k]$ for some $k \in [m]$ and for each $\ell \in [k]$ let $a_\ell \in B_\ell$ be the smallest element and $b_\ell \in B_\ell$ be the second-smallest element in $B_\ell$.
        Again without loss of generality, we can assume that $b_1 < b_2 < \dots < b_k$.

        We now show that for each $1 \leq i < j \leq k$ we have $ij \in E(R)$, which together with the fact that $R$ doesn't contain a copy of $\overrightarrow{K_{4 \log n}}$ will give us $k < 4 \log n$.
        Indeed, for such $i$ and $j$ let $v = (a_i, b_i, b_j)$ and notice first that since $\phi(a_i), \phi(b_i) \in V_i$, $a_i \to v$ and $v \to b_i$ we must have that $\phi(v) \in V_i$.
        Moreover, we must have $\phi(v)\phi(b_j) \in E(T)$ and thus $ij \in E(R)$.
    \end{proof}

    By the two claims we get that
    \[
        n = \sum_{\ell=1}^{n_1} |B_\ell| \leq m + 4\log n \cdot n' < n,
    \]
    a contradiction. Thus $T$ does not contain a copy of $S_n$.
    Finally, we have
    \begin{align*}
        |V(T)| = m \cdot n'^{\log n' /100 \log \log n'} &\geq \frac{n}{10} \cdot \left( \frac{n}{40 \log n} \right)^{\frac{\log n - \log (40 \log n)}{100 \log \log n}} \\
        &\geq \frac{n}{10} n^{\frac{\log n}{100\log \log n} - \frac{1}{10}} \cdot e^{-\frac{(\log n) \cdot \log (40 \log n)}{100 \log \log n}} \\
        &\geq \frac{n}{10}n^{\frac{\log n}{100\log \log n} - \frac{1}{5}}\\
        &\geq n^{\frac{\log n}{100\log \log n}}
    \end{align*}
    and thus $\vv{r}(S_n) \geq n^{\log n /100 \log \log n}$.
\end{proof}
Note that Theorem \ref{theorem:degenerate_lower_bound} follows from Theorem \ref{theorem:1_2_subdivision_ramsey}, since the underlying graph of $S_n$ is $3$-degenerate for each $n$.

\end{document}